\documentclass[letterpaper, 10 pt,journal]{ieeetran}

 \IEEEpubidadjcol 

\usepackage{hyperref}
\usepackage{cite}

\usepackage{graphics} 
\usepackage{epsfig} 
\usepackage{amsmath} 
\usepackage{amssymb}  

\usepackage{algpseudocode}
\usepackage{algorithm}
\usepackage{epstopdf}

\usepackage{amsthm}

\newtheorem{lemma}{\bf Lemma}
\newtheorem{corollary}{\bf Corollary}

\newcommand{\prob}[1]{\mathrm{Prob}\left[#1\right]}
\newcommand{\expect}[1]{\mathbb{E}\left[#1\right]}
\newcommand{\tr}[1]{\mathrm{tr}\left(#1\right)}
 \renewcommand{\vec}[1]{\mathrm{vec}\left(#1\right)}

\usepackage{verbatim}

\usepackage{setspace}
\usepackage{tikz}
 \usetikzlibrary{arrows,patterns,mindmap,backgrounds,shadows}
\usetikzlibrary{backgrounds}
\usetikzlibrary{shapes,arrows,chains}

\begin{document}
\IEEEoverridecommandlockouts
\IEEEpubid{\begin{minipage}{\textwidth}\ \\[12pt] \\ \\
         \copyright 2022 IEEE.  Personal use of this material is  permitted.  Permission from IEEE must be obtained for all other uses, in  any current or future media, including reprinting/republishing this material for advertising or promotional purposes, creating new  collective works, for resale or redistribution to servers or lists, or  reuse of any copyrighted component of this work in other works.
     \end{minipage}}

\title{\LARGE \bf State space models vs. multi-step predictors in predictive control:\\
Are state space models complicating safe data-driven designs?} 
\author{Johannes K\"ohler$^1$, Kim P. Wabersich$^1$,  Julian Berberich$^2$, Melanie N. Zeilinger$^1$%
\thanks{$^1$Institute for Dynamic Systems and Control, ETH Zürich, Zürich CH-8092, Switzerland (e-mail:
[jkoehle@ethz.ch).}
\thanks{$^2$Institute for Systems Theory and Automatic Control, University of Stuttgart, 70550 Stuttgart, Germany.}
}
\maketitle
%
\begin{abstract}
This paper contrasts recursive state space models and direct multi-step predictors for linear predictive control. 
We provide a tutorial exposition for both model structures to solve the following problems: 1. stochastic optimal control; 2. system identification; 3. stochastic optimal control based on the estimated model. 
Throughout the paper, we provide detailed discussions of the benefits and limitations of these two model parametrizations for predictive control and highlight the relation to existing works.  
Additionally, we derive a novel (partially tight) constraint tightening for stochastic predictive control with parametric uncertainty in the multi-step predictor.
\end{abstract} 
\section{Introduction}
\label{sec:intro} 
Model predictive control (MPC) is an optimization-based control strategy, which is applicable to general MIMO systems and 
directly accounts for state and input constraints~\cite{rawlings2017model}. 
Typically, first a parametric prediction model is identified.
In addition to noise and disturbances, the resulting parametric error then needs to be considered to ensure satisfaction of safety critical constraints. 
In this paper, we study data-driven predictive control problems using two different model parametrizations: state space models and \textit{multi-step predictors}, i.e., models that skip the sequential state propagation and directly predict $k$-steps into the future.
\subsubsection*{Related work} 
Historically, MPC emerged from the process control industry using impulse/step response models, which are simple to identify/adapt~\cite{richalet1978model,cutler1980dynamic}. 
These approaches have early been extended to robustly account for the parametric uncertainty of such finite impulse response (FIR) models~\cite{zheng1993robust}. 
However, the underlying assumptions result in intrinsic limitations~\cite{lundstrom1995limitations}, and hence, 
over the last three decades, state space models have become the de facto standard in (theoretical) research on MPC~\cite{mayne2000constrained,kouvaritakis2016model,rawlings2017model}. 
A variety of results are available to ensure robustness w.r.t. parametric uncertainty in state space models by using \textit{tube-based} approaches~\cite{primbs2009stochastic,fleming2014robust,lorenzen2019robust,Koehler2019Adaptive,lu2021robust,kouvaritakis2016model,schwenkel2021model}. 
More recently, a number of predictive control approaches have been proposed that directly exploit \textit{multi-step predictors}, as also done in early MPC approaches~\cite{richalet1978model,cutler1980dynamic,zheng1993robust}:
\begin{itemize}
\item dual/adaptive MPC using FIR models or more general orthonormal basis functions (OBFs~\cite{heuberger1995generalized}) in~\cite{tanaskovic2014adaptive,heirung2017dual,soloperto2019dual};
\item \textit{data-driven/enabled predictive control} based on an implicit model characterization using Hankel matrices~\cite{berberich2020data,coulson2021distributionally,huang2021robust,dorfler2021bridging}, compare also~\cite{yin2021maximum,yin2021data,iannelli2021experiment}; 
\item set-membership estimation for multi-step predictors in MPC~\cite{terzi2019learning,lauricella2020data,terzi2022robust};
\item feedback optimization in MPC using \textit{system responses}~\cite{chen2020robust,furieri2021near}, compare also~\cite{bujarbaruah2020robust};
\item nonlinear MPC using parametric~\cite{gruber2013computationally,genecili1995design} and non-parametric~\cite{maddalena2021kpc,thorpe2022data} multi-step predictors.
\end{itemize}
The main commonality is that multi-step predictors allow for simple bounds on the prediction error, compare Section~\ref{sec:data_driven_4} for a more detailed discussion of the individual approaches. 
More generally, forecasting problems (e.g., based on machine learning) are similarly divided between \textit{sequential/recursive} strategies and \textit{direct/independent} strategies~\cite{taieb2012review}.
\subsubsection*{Contribution}
We consider a stochastic optimal control problem for linear systems, as typically arising in MPC, and provide a tutorial-style exposition based on state space models and multi-step predictors, respectively.\footnote{%
Although MPC relies on a receding horizon implementation, we initially focus on the open-loop problem to simplify the exposition. Closed-loop implementations with corresponding caveats are discussed in Section~\ref{sec:data_driven_3}.} 
First, we reformulate this problem as equivalent quadratic programs (QPs) for both parametrizations (Sec.~\ref{sec:problem}).
Then, we use a maximum likelihood estimate (MLE) for the system identification (Sec.~\ref{sec:SysID}). 
Finally, we study the stochastic optimal control problem with uncertain parameter estimates and derive a novel data-driven stochastic MPC formulation (Sec.~\ref{sec:data_driven}). 
Therein, we demonstrate that multi-step predictors allow for a simpler (partially tight) reformulation to account for the parametric uncertainty.
On the other hand, accounting for parametric uncertainty in state space models typically requires a \textit{sequential propagation}~\cite{primbs2009stochastic,fleming2014robust,lorenzen2019robust,Koehler2019Adaptive,lu2021robust,kouvaritakis2016model}, which can result in significant conservatism. 
Throughout the paper, we provide extensive discussions on the advantages and limitations of the two parametrizations (Sec.~\ref{sec:problem_4}, \ref{sec:SysID_3}, \ref{sec:data_driven_3}), and the relation to existing work (Sec.~\ref{sec:data_driven_4}).
\IEEEpubidadjcol
\subsubsection*{Notation} 
The probability of an event $A$ and the expected value of a random variable $x$ are denoted by $\prob{A}$ and $\expect{x}$, respectively. 
We denote a variable $w$ with Gaussian distribution with mean $\mu$ and variance $\Sigma$ by $w\sim\mathcal{N}(\mu,\Sigma)$. 
The quantile function of the chi-squared distribution with $n$ degrees of freedoms is given by $\chi_n^2(p)$ with probability $p\in[0,1]$. 
The set of integers in the interval $[a,b]\subseteq\mathbb{R}$ is denoted by $\mathbb{I}_{[a,b]}$. 
For a sequence of matrices $x_t\in\mathbb{R}^{n\times m}$, $t\in\mathbb{I}_{[t_1,t_2]}$, $t_1,t_2\in\mathbb{I}_{\geq 0}$, we denote the stacked matrix by $x_{[t_1,t_2]}\in\mathbb{R}^{n(t_2+1-t_1)\times m}$. 
The trace of a square matrix $A$ is denoted by $\tr{A}$. 
By $\vec{A}\in\mathbb{R}^{nm}$ we denote a vector that stacks the columns of the matrix $A\in\mathbb{R}^{n\times m}$. 
The Kronecker product is denoted by $\otimes$. 
The identity matrix is given by $I_n\in\mathbb{R}^{n\times n}$ and a matrix of zeros by $0_{n\times m}\in\mathbb{R}^{n\times m}$. 
We denote a block-diagonal matrix with square matrices $A\in\mathbb{R}^{n\times n}$ on its block-diagonal by $\mathrm{diag}_k(A)\in\mathbb{R}^{kn\times kn}$, $k\in\mathbb{I}_{\geq 0}$. 
Positive definiteness of a matrix $Q$ is denoted by $Q\succ 0$.
We denote the symmetric matrix square-root by $A^{1/2}$ for $A\succ 0$.
We write $\|x\|_Q^2=x^\top Q x$ for $Q\succ 0$.

\section{Stochastic predictive control}
\label{sec:problem}
We first state the control problem (Sec.~\ref{sec:problem_1}) and convert it into a deterministic QP (Sec.~\ref{sec:problem_2}). 
Then, we introduce the multi-step predictors, derive an equivalent QP (Sec.~\ref{sec:problem_3}), and provide a discussion (Sec.~\ref{sec:problem_4}). 

\subsection{Problem setup}
\label{sec:problem_1}
We consider a linear discrete-time system of the form
\begin{align}
\label{eq:sys_state}
x_{k+1}=Ax_k+Bu_k+Ew_k,\quad k\in\mathbb{I}_{\geq 0},
\end{align}
with state $x_k\in\mathbb{R}^n$, input $u_k\in\mathbb{U}\subseteq\mathbb{R}^m$, disturbances $w_k\in\mathbb{R}^q$, time $k\in\mathbb{I}_{\geq 0}$, and input constraint set $\mathbb{U}$. 
The disturbances are i.i.d. Gaussian variables $w_k\sim\mathcal{N}(0,\Sigma_{\mathrm{w}})$ and the initial state is uncertain with $x_0\sim\mathcal{N}(\overline{x}_0,\Sigma_{\mathrm{x},0})$, which is independent of $w_k$. 
We consider the following stochastic optimal control problem
\begin{subequations}
\label{eq:stoch_opt}
\begin{align}
&\min_{u_{[0,N-1]}\in\mathbb{U}^N}\expect{\sum_{k=0}^{N-1}\|x_{k+1}\|_Q^2+\|u_k\|_R^2}\\
\label{eq:chance_constraints}
\text{s.t. }&\prob{H_{\mathrm{x},j}^\top x_k\leq 1}\geq p,~ j\in\mathbb{I}_{[1,r]},~ k\in\mathbb{I}_{[0,N]},
\end{align}
\end{subequations}
with $Q,R\succ 0$, $H_{\mathrm{x},j}\in\mathbb{R}^n$, $j\in\mathbb{I}_{[1,r]}$, and probability $p\in(0,1)$ for the chance constraints~\eqref{eq:chance_constraints}. 
We consider the \textit{stochastic} and \textit{open-loop} control problem to allow for a simple system identification using MLE (Sec.~\ref{sec:SysID}) and an easier comparison between the two model parametrizations, respectively. 
Generalizations of this problem setup (e.g., receding horizon, robust, input-output models) are discussed in Sections~\ref{sec:SysID_3}, \ref{sec:data_driven_4}, \ref{sec:data_driven_3}. 
In order to introduce the problem, we assume that the model parameters are known. 
The problem of system identification and predictive control with probabilistic bounds on the model parameters is treated in Sections~\ref{sec:SysID} and \ref{sec:data_driven}, respectively. 

\subsection{State space model} 
\label{sec:problem_2} 
Problem~\eqref{eq:stoch_opt} is commonly used in stochastic MPC and we derive a deterministic formulation following~\cite[Sec.~3.2.1]{farina2016stochastic}, \cite{hewing2020recursively}. 
For a given input sequence $u_{[0,N-1]}\in\mathbb{U}^N$ and initial state $x_0\in\mathcal{N}(\overline{x}_0,\Sigma_{\mathrm{x},0})$, the predicted states are Gaussian variables $x_k\sim\mathcal{N}(\overline{x}_k,\Sigma_{\mathrm{x},k})$, $k\in\mathbb{I}_{[1,N]}$. 
The mean and variance can be recursively computed as $\overline{x}_{k+1}=A\overline{x}_k+Bu_k$ and $\Sigma_{\mathrm{x},k+1}=A \Sigma_{\mathrm{x},k}A^\top+E\Sigma_{\mathrm{w}}E^\top$, $k\in\mathbb{I}_{[0,N-1]}$. 
The probabilistic constraint~\eqref{eq:chance_constraints} is equivalent to $H_{\mathrm{x},j}\overline{x}_k\leq 1-c_p\|H_{\mathrm{x},j}\|_{\Sigma_{\mathrm{x},k}}$, with $c_p:=\sqrt{\chi_1^2(2p-1)}$. Furthermore, we have $\expect{\|x_k\|_Q^2}=\|\overline{x}_k\|_Q^2+\tr{Q\Sigma_{\mathrm{x},k}}$. 
By pre-computing the variance $\Sigma_{\mathrm{x},k}$ and neglecting the constant terms in the cost, Problem~\eqref{eq:stoch_opt} is equivalent to the following QP:
\begin{align}
\label{eq:stoch_opt_state}
&\min_{u_{[0,N-1]}\in\mathbb{U}^N}~\sum_{k=0}^{N-1}\|\overline{x}_{k+1}\|_Q^2+\|u_k\|_R^2\\
\text{s.t. }&H_{\mathrm{x},j}^\top\overline{x}_k\leq 1-c_p\|H_{\mathrm{x},j}\|_{\Sigma_{\mathrm{x},k}} 
,~ j\in\mathbb{I}_{[1,r]},~k\in\mathbb{I}_{[0,N]},\nonumber\\
&\overline{x}_{k+1}=A\overline{x}_k+Bu_k,~k\in\mathbb{I}_{[0,N-1]}.\nonumber
\end{align}

\subsection{Multi-step predictors}
\label{sec:problem_3}
\textit{Multi-step predictors} are given by 
\begin{align}
\label{eq:model_convoluted}
x_{[1,N]}=G_{\mathrm{u}}u_{[0,N-1]}+G_{0}x_0+G_{\mathrm{w}}w_{[0,N-1]},
\end{align}
where the matrices $G_{\mathrm{u}},G_0,G_{\mathrm{w}}$ directly map  to the future state sequence, instead of sequentially applying the one-step prediction of the state space model~\eqref{eq:sys_state}. 
The rows in~\eqref{eq:model_convoluted} correspond to different horizon steps $k\in\mathbb{I}_{[1,N]}$ with 
\begin{align}
\label{eq:model_convoluted_split}
&x_k=G_{0,k}x_0+G_{\mathrm{u},k}u_{[0,k-1]}+G_{\mathrm{w},k}w_{[0,k-1]},~G_{0,k}=A^k,\nonumber\\ 
&G_{\mathrm{u},k}=
\begin{bmatrix}
A^{k-1}B,\dots,B
\end{bmatrix},~
G_{\mathrm{w},k}
=\begin{bmatrix}
A^{k-1}E,\dots, E
\end{bmatrix}.
\end{align}
This model corresponds to a compressed/condensed version of the recursive application of the state space model~\eqref{eq:sys_state}, as is often used for the numerical solution of the QP~\eqref{eq:stoch_opt_state} (cf.~\cite[Chap.~7, Sec.~8.8.4]{rawlings2017model}). 
This model structure is prevalent in MPC schemes based on FIR models~\cite{richalet1978model,cutler1980dynamic,zheng1993robust,tanaskovic2014adaptive} or subspace predictive control (SPC)~\cite{favoreel1999spc}, and is recently employed in many (robust) predictive control approaches~\cite{tanaskovic2014adaptive,heirung2017dual,soloperto2019dual,berberich2020data,coulson2021distributionally,huang2021robust,dorfler2021bridging, yin2021maximum,yin2021data,iannelli2021experiment,terzi2019learning,lauricella2020data,terzi2022robust,furieri2021near,chen2020robust,bujarbaruah2020robust,maddalena2021kpc,thorpe2022data} (cf. discussion Section~\ref{sec:data_driven_4}).  

Considering the model~\eqref{eq:model_convoluted_split} with initial state $x_0\sim\mathcal{N}(\overline{x}_0,\Sigma_{\mathrm{x},0})$, we obtain $x_k\sim\mathcal{N}(\overline{x}_k,\Sigma_{\mathrm{x},k})$ with 
$\overline{x}_k=G_{\mathrm{u},k}u_{[0,k-1]}+G_{0,k}\overline{x}_0$ and
$\Sigma_{\mathrm{x},k}=G_{0,k}\Sigma_{\mathrm{x},0}G_{0,k}^\top
+G_{\mathrm{w},k}\mathrm{diag}_k(\Sigma_{\mathrm{w}})G_{\mathrm{w},k}^\top$, $k\in\mathbb{I}_{[1,N]}$. 
Thus, we can reformulate Problem~\eqref{eq:stoch_opt} as the following QP: 
\begin{align}
\label{eq:stoch_opt_convoluted}
&\min_{u_{[0,N-1]}\in\mathbb{U}^N}~\sum_{k=0}^{N-1}\|\overline{x}_{k+1}\|_Q^2+\|u_k\|_R^2\\
\text{s.t. }&H_{\mathrm{x},j}^\top\overline{x}_k\leq 1-c_p\|H_{\mathrm{x},j}\|_{\Sigma_{\mathrm{x},k}},~ 
 j\in\mathbb{I}_{[1,r]},~k\in\mathbb{I}_{[0,N]},\nonumber\\
&\overline{x}_{k}=G_{0,k}\overline{x}_0+G_{\mathrm{u},k}u_{[0,k-1]},~k\in\mathbb{I}_{[1,N]}.\nonumber
\end{align}

\subsection{Discussion}
\label{sec:problem_4}
First, note that both optimization problems are exact solutions to the stochastic optimal control problem. 
\begin{lemma}
\label{lemma:stoch_eqiv}
The stochastic optimal control problem~\eqref{eq:stoch_opt}, and the two deterministic QPs~\eqref{eq:stoch_opt_state} and~\eqref{eq:stoch_opt_convoluted} result in the same minimizing input sequence $u_{[0,N-1]}$.
\end{lemma}
With the state space model, the mean and variance are sequentially computed by applying a one-step discrete-time model, while Equation~\eqref{eq:model_convoluted_split} directly yields the multi-step prediction.   
Problem~\eqref{eq:stoch_opt_state} has a sparse structure, which is computationally more efficient for long horizons $N$ (cf.~\cite[Sec.~3]{kouzoupis2018recent}).  
Given that all the matrices are available offline, we can also derive Problem~\eqref{eq:stoch_opt_convoluted} by condensing Problem~\eqref{eq:stoch_opt_state}, i.e., by eliminating the dynamics equations through substitution.

Beyond the stochastic control problem~\eqref{eq:stoch_opt}, the state space formulation also allows for an easy computation of control Lyapunov functions and invariant sets, which are typically employed in a receding horizon MPC implementation~\cite{rawlings2017model}. 
Thus, in case of a given system model, there is no clear drawback in using the state space model.  
This also explains why many modern MPC textbooks (cf., e.g.,~\cite{kouvaritakis2016model,rawlings2017model}) consider only state space models.  
However, we discuss later (Sec.~\ref{sec:data_driven}) that parametric uncertainties in state space models can significantly complicate the design, which is not necessarily the case for multi-step predictors. 

\section{System Identification}
\label{sec:SysID} 
In the following, we assume that the model parameters are unknown and need to be identified.
To this end, we apply some probing input signal $u_{[0,T-1]}$ over a time period $T\in\mathbb{I}_{\geq 1}$ to the system~\eqref{eq:sys_state} and obtain noisy state measurements $\tilde{x}_k=x_k+\epsilon_k$, $k\in\mathbb{I}_{[0,T]}$ with i.i.d. noise $\epsilon_k\sim\mathcal{N}(0,\Sigma_\epsilon)$. 
Generalizations to noisy output measurements are discussed in Section~\ref{sec:data_driven_3}). 
We assume that the matrices characterizing the variance of the disturbances, i.e., $\Sigma_{\mathrm{w}},\Sigma_\epsilon,E,G_{\mathrm{w}}$, are known. 
First, we focus on the identification of a multi-step predictor (Sec.~\ref{sec:SysID_1}). 
Then, we consider the special case of state space models (Sec.~\ref{sec:SysID_2}) and contrast the results (Sec.~\ref{sec:SysID_3}). 

\subsection{Multi-step predictors}
\label{sec:SysID_1}
As noted early~\cite{favoreel1999spc}, subspace identification computes the matrices $G_{\mathrm{u}}$, $G_0$ as intermediate quantities and hence there is no need to identify a state space model. 
The direct identification of the multi-step predictor~\eqref{eq:model_convoluted_split} can be written as a linear regression problem with the parameters $\theta_k:=\vec{[G_{0,k},G_{\mathrm{u},k}]}\in\mathbb{R}^{n^2+nkm}$. 
The following lemma provides an MLE to identify each predictor $\theta_k$ individually. 
\begin{lemma}
\label{lemma:id_correlated}
Suppose we have a uniform prior over the parameters $\theta_k$. 
Consider the noisy regressor $\tilde{\Phi}^k_j:=[\tilde{x}_j^\top,u_{[j,j+k-1]}^\top]\otimes I_n\in\mathbb{R}^{n\times nkm+n^2}$, the covariance matrix $\tilde{\Sigma}_{\mathrm{w,corr},k}$ according to~\eqref{eq:Sigma_corr_conv} below, and suppose that $\Sigma_{\theta,k}^{-1}:=\tilde{\Phi}^{k,\top}_{[0,T-k]}\tilde{\Sigma}_{\mathrm{w,corr},k}^{-1}\tilde{\Phi}^k_{[0,T-k]}\succ 0$.\footnote{%
In case $\tilde{\Sigma}_{\mathrm{w,corr},k}$ is singular, the identification problem should be projected on the corresponding subspace.} 
Then, for any $\delta\in(0,1)$, we have $\prob{\theta_k-\hat{\theta}_k\in\Theta_{\delta,k}}\geq \delta$ with
\begin{align}
\label{eq:ellipse}
\Theta_{\delta,k}=\{\theta|~\theta^\top \Sigma_{\theta,k}^{-1}\theta\leq \chi^2_{nkm+n^2}(\delta)\},
\end{align}
and the weighted least-squares estimate
\begin{align}
\label{eq:LS_corr}
\hat{\theta}_k:=&\arg\min_{\theta}\|\tilde{x}_{[k,T]}-\tilde{\Phi}^k_{[0,T-k]}\theta\|_{\tilde{\Sigma}_{\mathrm{w,corr},k}^{-1}}^2\\
=&\Sigma_{\theta,k}\tilde{\Phi}^{k,\top}_{[0,T-k]}\tilde{\Sigma}_{\mathrm{w,corr},k}^{-1}\tilde{x}_{[k,T]}.\nonumber
\end{align}
\end{lemma}
\begin{proof}
This result is inspired by the MLE in~\cite[Lemma 3.1]{umenberger2019robust} and classical techniques to account for correlated data~\cite[Sec.~5]{aastrom1971system}. 
Following the steps in~\cite[A.1.1]{umenberger2019robust}, the dynamics~\eqref{eq:model_convoluted_split} satisfy
$x_{j+k}=\Phi^k_j\theta_k+G_{\mathrm{w},k}w_{[j,j+k-1]}$
with the regressor $\Phi^k_j=[x_j^\top,u_{[j,j+k-1]}^\top]\otimes I_n$. 
Given the noisy state $\tilde{x}_{k+j}=x_{k+j}+\epsilon_{k+j}$ and regressor $\tilde{\Phi}^k_j=\Phi^k_j+[\epsilon_j^\top,0_{1\times km}]\otimes I_n$, we have
$\tilde{x}_{k+j}=\tilde{\Phi}^k_j\theta_k+\tilde{w}^k_j$, 
with the residual $\tilde{w}^k_j=G_{\mathrm{w},k}w_{[j,j+k-1]}-G_{0,k}\epsilon_j+\epsilon_{k+j}$.
By stacking these regression equations we get
$\tilde{x}_{[k,T]}=\tilde{\Phi}^k_{[0,T-k]}\theta_k+\tilde{w}_{[0,T-k]}$. 
The residual $\tilde{w}_{[0,T-k]}$ is a linear combination of independent Gaussian variables $w_k$, $\epsilon_k$ and hence also follows a Gaussian distribution, i.e., $\tilde{w}_{[0,T-k]}\sim\mathcal{N}(0,\tilde{\Sigma}_{\mathrm{w,corr},k})$, with covariance $\tilde{\Sigma}_{\mathrm{w,corr},k}=\expect{\tilde{w}^k_{[0,T-k]}\tilde{w}_{[0,T-k]}^{k,\top}}$. 
The matrix $\Sigma_{\mathrm{w,corr},k}$ is a band-diagonal matrix with the elements
\begin{align}
\label{eq:Sigma_corr_conv}
\expect{\tilde{w}^k_j\tilde{w}^{k,\top}_j}=&G_{\mathrm{w},k}\mathrm{diag}_k(\Sigma_{\mathrm{w}})G_{\mathrm{w},k}^\top+G_{0,k}\Sigma_{\epsilon}G_{0,k}^\top+\Sigma_\epsilon,\nonumber\\
\expect{\tilde{w}^k_j\tilde{w}^{k,\top}_{i+j}}=&
G_{\mathrm{w},k}\begin{bmatrix}
0_{i\cdot q\times (k-i)q}&0_{i\cdot q\times i\cdot q}\\
\mathrm{diag}_{k-i}(\Sigma_{\mathrm{w}})&0_{(k-i)q\times iq} 
\end{bmatrix}G_{\mathrm{w},k}^\top,\nonumber\\
&i\in\mathbb{I}_{[1,k-1]},\nonumber\\
\expect{\tilde{w}^k_j\tilde{w}^{k,\top}_{j+k}}=&- \Sigma_\epsilon G_{0,k}^\top,~
\expect{\tilde{w}^k_j\tilde{w}^{k,\top}_{j+i}}=0_q,~ i>k.
\end{align} 
The likelihood is given by
\begin{align*}
&\prob{\tilde{w}^k_{[0,T-k]}=\tilde{x}_{[k,T]}-\tilde{\Phi}^k_{[k,T-k]}\theta_k}\\
\propto& \exp\left[-\dfrac{1}{2}\|\tilde{x}_{[k,T]}-\tilde{\Phi}^k_{[0,T-k]}\theta_k\|_{\tilde{\Sigma}_{\mathrm{w,corr},k}^{-1}}^2\right]\\
\stackrel{\eqref{eq:LS_corr}}{\propto}&\exp\left[-\dfrac{1}{2}\|\theta_k-\hat{\theta}_k\|_{\tilde{\Sigma}_{\theta,k}^{-1}}^2\right],
\end{align*}
and hence the weighted least-squares estimate~\eqref{eq:LS_corr} yields the MLE.
The uniform prior over $\theta_k$ ensures that the posterior distribution is proportional to the likelihood (cf.~\cite[Equ.~(18)]{umenberger2019robust}) and hence
 $\theta_k\sim\mathcal{N}(\hat{\theta}_k,\Sigma_{\theta,k})$. 
The set~\eqref{eq:ellipse} follows from the multivariate Gaussian distribution. 
\end{proof}
This result provides a simple identification method in terms of a weighted least-squares estimate and provides a high-probability bound on the parameter error {$\tilde{\theta}_k:=\theta_k-\hat{\theta}_k$}. 
However, to directly apply the result we need to know the covariance matrix $\tilde{\Sigma}_{\mathrm{w,corr},k}$, which explicitly depends on $G_{0,k}$, i.e., the impulse response to be identified. 
This is comparable to the \textit{generalized least-squares} method, where the correlation information is used to filter the data prior to applying the least-squares estimate (cf.~\cite[Sec.~5]{aastrom1971system}).
As this information is typically not known (exactly) a-priori, the corresponding parameters can be jointly optimized in the MLE, 
resulting in a highly nonlinear optimization problem, which can be solved using the \textit{expectation-maximisation} technique~\cite{gibson2005robust}. 
Note that neglecting this correlation results in a bias in the least-squares estimate, compare also the overview on linear system identification~\cite[Sec.~5]{aastrom1971system}. 

For long predictions $k\gg 1$, this result simplifies if we have only measurement noise $\epsilon_k$ and no disturbances $w_k$. 
In this case, the covariance $\Sigma_{\mathrm{w,corr}}$ has only two off-diagonal terms corresponding to $-\Sigma_{\epsilon}G_{0,k}$. 
For open-loop stable systems, we have $\lim_{k\rightarrow\infty}\|G_{0,k}\|=0$ or in case of FIR models we directly have $G_0=0$. 
Hence, in case of open-loop systems without additive disturbances, the bias of a simple least-squares estimation tends to zero for $k\rightarrow\infty$.

\subsection{State space model}
\label{sec:SysID_2}
Note that $\theta_1=\vec{[G_{0,1},G_{\mathrm{u},1}]}=\vec{[A,B]}\in\mathbb{R}^{n(n+m)}$, i.e., the MLE in Lemma~\ref{lemma:id_correlated} also covers the state space problem as a special case. 
Specifically, in this case the structure of the corresponding covariance matrix simplifies with $\tilde{\Sigma}_{\mathrm{w}}=E\Sigma_{\mathrm{w}}E^\top+\Sigma_\epsilon+A\Sigma_\epsilon A^\top$ and 
\begin{align}
\label{eq:Sigma_corr_state}
\tilde{\Sigma}_{\mathrm{w,corr},1}:=
\begin{bmatrix}
\tilde{\Sigma}_{\mathrm{w}}&-\Sigma_{\epsilon}A^\top&0_{n\times n(T-2)}\\ 
-A\Sigma_{\epsilon}&\tilde{\Sigma}_{\mathrm{w}}&\ddots\\
0_{n(T-2)\times n}&\ddots&\ddots\\
\end{bmatrix}, 
\end{align}
The following corollary shows that we can use a simple least-squares estimate for state space models if we assume noise-free state measurements.
\begin{corollary}
\label{cor:id_state}
Suppose we have a uniform prior over the parameters $\theta_1$ and noise-free measurements ($\Sigma_{\epsilon}=0$). 
Consider the regressor $\Phi_k^1=[x_k^\top,u_k^\top]\otimes I_n\in\mathbb{R}^{n(n+m)}$ and suppose $\Sigma_{\theta,1}^{-1}=\sum_{k=0}^{T-1}\Phi_k^{1,\top} \tilde{\Sigma}_{\mathrm{w}}^{-1}\Phi^1_k\succ 0$ with $\tilde{\Sigma}_{\mathrm{w}}:=E\Sigma_{\mathrm{w}}E^\top$. 
Then, for any $\delta\in(0,1)$, we have $\prob{\theta_1-\hat{\theta}_1\in\Theta_{\delta,1}}\geq \delta$ with $\Theta_{\delta,1}$ according to~\eqref{eq:ellipse} and the least-squares estimate
\begin{align}
\label{eq:LS_noiseFree}
\hat{\theta}_1=\arg\min_{\theta}\sum_{k=0}^{T-1}\|{x}_{k+1}-\Phi^1_k\theta\|_{\tilde{\Sigma}_{\mathrm{w}}^{-1}}^2.
\end{align}
\end{corollary} 
\begin{proof}
For $\Sigma_{\epsilon}=0$, the matrix $\tilde{\Sigma}_{\mathrm{w,corr},1}$ is block-diagonal with elements $\tilde{\Sigma}_{\mathrm{w}}$. 
Hence, the least-squares estimate and the covariance matrix can be equivalently written as a sum.
\end{proof}
This result essentially recovers the MLE bounds from~\cite[Prop.~2.1]{umenberger2019robust}. 
The ellipsoidal bound on the parameters $\theta$ also implies an ellipsoidal bound on the system matrices (cf.~\cite[Lemma.~3.1]{umenberger2019robust}), which may be easier to use for robust control.

\subsection{Discussion}
\label{sec:SysID_3}
We formulated the identification problem as an MLE with a weighted least-squares estimate, resulting in a (high-probability) bound on the parameter error.
However, the presented bounds and estimates can only be applied if the correlation in the disturbances is known. 
In general, a joint identification of the parameters and this covariance matrix is needed, which results in a nonlinear problem (cf.~{\cite[Sec.~5]{aastrom1971system}), which can be solved using \textit{expectation-maximisation}~\cite{gibson2005robust}.

\subsubsection*{Scalability, conditioning, prior}
A crucial difference between the two parametrizations is the number of required parameters. 
In particular, for the state space model we have $\theta\in\mathbb{R}^{n^2+nm}$ and the parameters often directly relate to physical quantities (e.g., damping constant). 
This allows for a meaningful integration of priors on the parameters, which can reduce the sample complexity. 

On the other hand, 
directly identifying a multi-step predictor requires $\theta_k\in\mathbb{R}^{nmk+n^2}$, i.e., the number of parameters scales linearly with the horizon $k$. 
This issue can be relaxed by considering truncated FIR models (cf.~\cite{lundstrom1995limitations}, \cite[App.~A]{tanaskovic2014adaptive}) or structured parametrizations using OBF~\cite{heuberger1995generalized}. 
Multi-step predictors allow to naturally impose general open-loop stability priors, e.g., in terms of over-shoot constants and decay rates (without specifying a Lyapunov function).

Hence, in case of unstable/marginally-stable systems or physical priors on model parameters, identifying a state space model has benefits. 
On the other hand, the direct identification of multi-step predictors allows to encode priors in terms of open-loop stability. 
For long horizons, multi-step predictors can provide better accuracy (cf.~\cite{shook1991identification}, \cite[Thm.~1]{terzi2022robust}) at the cost of an increased complexity.  

\subsubsection*{Impact of noise or disturbances on identification} 
The state space identification problem becomes simpler if we assume noise-free state measurements (cf. Cor.~\ref{cor:id_state}). 
On the other hand, the identification of multi-step predictors (Lemma~\ref{lemma:id_correlated}) for long horizons ($k\gg 1$) simplifies if only measurement noise is present.  
For comparison, early MPC approaches based on impulse responses, e.g. dynamic matrix control~\cite{cutler1980dynamic}, are restricted to "disturbances on the output"~\cite{lundstrom1995limitations}. 
OBFs~\cite{heuberger1995generalized}, a structured parametrization of~\eqref{eq:model_convoluted}, also allow for a simple identification under measurement noise~\cite{heirung2017dual,soloperto2019dual}. 
The identification of multi-step predictors is also related to the (implicit) data-driven models based on Hankel matrices~\cite{coulson2021distributionally,berberich2020data,dorfler2021bridging,huang2021robust,yin2021maximum,yin2021data,iannelli2021experiment}, which typically also consider noisy output measurements without disturbances.  
A corresponding MLE with an ellipsoidal confidence bound was recently derived in~\cite{yin2021maximum,yin2021data}, which is also applicable to unbiased impulse response estimation~\cite[Sec.~6]{yin2021maximum}. 
In particular, this approach also considers only output measurement noise without disturbances, the MLE is nonlinear including non-diagonal weights to account for cross correlation, which require prior knowledge of $G_{0,k}$. 
In~\cite[Sec.~3.4]{furieri2021near}, a multi-step predictor~\eqref{eq:model_convoluted} subject to both noisy measurements and disturbances is identified using simple a least-squares estimate or the MLE~\cite{yin2021maximum}, however, without deriving corresponding error bounds.

We note that by using non-overlapping data-segments (structured in a page matrix), the cross correlation in Lemma~\ref{lemma:id_correlated} can be removed. 
For the special case of state space models, this corresponds to skipping every second measurement. 
In both parametrizations, the resulting uncorrelated residuals allow for a simple least-squares estimate.

\subsubsection*{Robust identification methods}
The challenge inherent in MLE with correlated disturbances can also be avoided by considering \textit{robust} identification methods. 
In particular, given a compact bound on the disturbances, set-membership estimation can be used to identify the \textit{non-falsified} set that contains the true parameters, which is frequently used in robust adaptive MPC formulations~\cite{tanaskovic2014adaptive,terzi2019learning,terzi2022robust,lorenzen2019robust,lu2021robust,Koehler2019Adaptive}. 
In particular, the approaches in~\cite[Sec.~5.4]{lu2021robust} and \cite{terzi2019learning} are applicable to both noisy measurements and disturbances; \cite[Cor.~3]{lu2021robust} provides convergence results for fixed complexity estimates; and \textit{multi-step predictors}~\eqref{eq:model_convoluted_split} are explicitly considered in~\cite{terzi2019learning,terzi2022robust}.   
Note that most of the discussion in this paper regarding predictive control based on multi-step predictors equally applies to such a robust setup.

\subsubsection*{Input experiment design}
The derived estimation bounds are such that collecting more (informative) data reduces the parametric uncertainty. 
A related problem is choosing the probing input signal $u_{[0,T-1]}$, which is studied in dual control or optimal experiment design. 
Given, e.g., Corollary~\ref{cor:id_state}, a natural approach is to ensure a lower bound on the persistence of excitation, i.e., $\sum_{k=0}^{T-1} \Phi_k^\top \tilde{\Sigma}_{\mathrm{w}}^{-1}\Phi_k\succeq \beta I$, given a rough model description. 
This issue is, e.g., addressed in~\cite{umenberger2019robust,lu2021robust}, which introduce conservatism due to convex relaxations, and only obtain approximations~\cite{umenberger2019robust} or require more complex robust propagation methods~\cite[Sec~3.4]{lu2021robust}. 
On the other hand, if the problem is addressed for multi-step predictors~\eqref{eq:model_convoluted} with only measurement noise, the variance can be deterministically predicted, which is, e.g., exploited in~\cite{heirung2017dual,soloperto2019dual,iannelli2021experiment}.

\section{Data-driven predictive control with parametric uncertainty}  
\label{sec:data_driven}
In the following, we revisit the stochastic optimal control problem (Sec.~\ref{sec:problem}) given the uncertain parameter estimates (Sec.~\ref{sec:SysID}). 
To this end, we use the concept of \textit{robustness in probability}, i.e., we consider $\tilde{\theta}_k\in\Theta_{\delta,k}$, $\delta>p$ (Lemma~\ref{lemma:id_correlated}) robustly and enforce the chance constraints~\eqref{eq:chance_constraints} with a larger probability $\tilde{p}=p/\delta\in(p,1)$ (cf. \cite[Lemma~V.3]{wabersich2021probabilistic}). 
We first focus on the state space model (Sec.~\ref{sec:data_driven_1}) and then the multi-step predictors (Sec.~\ref{sec:data_driven_2}). 
Finally, we discuss related work (Sec.~\ref{sec:data_driven_4}) and contrast the results (Sec.~\ref{sec:data_driven_3}). 

\subsection{State space model} 
\label{sec:data_driven_1}
The predicted states are Gaussian ${x}_k\sim\mathcal{N}(\overline{x}_k,\Sigma_{\mathrm{x},k})$ (Sec.~\ref{sec:problem}), however, $\overline{x}_k$ and $\Sigma_{\mathrm{x},k}$ depend on the uncertain parameters $\theta_1=\vec{[A,B]}$. 
Conceptually, we can solve Problem~\eqref{eq:stoch_opt} using the following problem 
\begin{align}
\label{eq:stoch_opt_robust}
&\min_{u_{[0,N-1]}\in\mathbb{U}^N}\max_{\tilde{\theta}_1\in\Theta_{\delta,1}}\sum_{k=0}^{N-1}\|\overline{x}_{k+1}\|_Q^2+\|u_k\|_R^2+\tr{Q\Sigma_{\mathrm{x},k+1}}\nonumber\\
\text{s.t. }&H_{\mathrm{x},j}^\top\overline{x}_k\leq 1-c_{\tilde{p}}\|H_{\mathrm{x},j}\|_{\Sigma_{\mathrm{x},k}},
 j\in\mathbb{I}_{[1,r]},k\in\mathbb{I}_{[0,N]},\\
&\Sigma_{\mathrm{x},k+1}=A \Sigma_{\mathrm{x},k}A^\top+E\Sigma_{\mathrm{w}}E^\top,~ k\in\mathbb{I}_{[0,N-1]},~\tilde{\theta}_1\in\Theta_{1,\delta},\nonumber\\
&\overline{x}_{k+1}=A\overline{x}_k+Bu_k,~k\in\mathbb{I}_{[0,N-1]},~\tilde{\theta}_1\in\Theta_{1,\delta},\nonumber
\end{align}
with $c_{\tilde{p}}:=\sqrt{\chi_1^2(2\tilde{p}-1)}$. 
This deterministic $\min-\max$ problem is a non-conservative solution to Problem~\eqref{eq:stoch_opt}, given $\prob{\tilde{\theta}_1\in\Theta_{1,\delta}}\geq \delta$. 
However, $\Sigma_{\mathrm{x},k}$ and $\overline{x}_k$ are highly nonlinear in $\tilde{\theta}_1$ and the problem is in general intractable.

We briefly mention standard approaches from the MPC literature to compute a feasible (but suboptimal) solution to~\eqref{eq:stoch_opt_robust}. 
By ignoring the time-invariance of $\theta$, an upper bound for $\Sigma_{\mathrm{x},k}$ can be computed offline(!) using semi-definite programs (SDP), compare, e.g., \cite[Lemma~4.1, Thm.~4.2]{umenberger2019robust}, \cite{primbs2009stochastic}. 
Determining the reachable set for the mean $\overline{x}_k$ is addressed in tube-based MPC formulations for multiplicative/parametric uncertainty~\cite{primbs2009stochastic,fleming2014robust,lorenzen2019robust,Koehler2019Adaptive,lu2021robust}, \cite[Sec.~5.5]{kouvaritakis2016model}.
These approaches use a fixed parametrization (polytopes/ellipses) and provide sufficient conditions for an over-approximation of the reachable set. 
The size of this "tube/funnel" depends on the nominal input $u_{[0,N-1]}$ and is computed by adding linear/conic constraints over the prediction horizon.

\subsection{Multi-step predictors}
\label{sec:data_driven_2}
The following lemma shows how to enforce the chance constraints~\eqref{eq:chance_constraints} despite the parametric uncertainty. 
\begin{lemma}
\label{lemma:convoluted_robust}
Suppose the conditions in Lemma~\ref{lemma:id_correlated} hold and the input sequence $u_{[0,k-1]}$ satisfies
\begin{align}
\label{eq:convoluted_tightening}
&H_{\mathrm{x},j}^\top(\hat{G}_{0,k}\overline{x}_0+\hat{G}_{\mathrm{u},k}u_{[0,k-1]})+c_{\tilde{p}}\overline{h}_{j,k}\\ 
&+\sqrt{\chi^2_{nkm+n^2}(\delta)}\left\|\Sigma_{\theta,k}^{1/2}\begin{bmatrix}
\overline{x}_0\\u_{[0,k-1]}
\end{bmatrix}\otimes H_{\mathrm{x},j}\right\|\leq 1,~j\in\mathbb{I}_{[1,r]},\nonumber
\end{align}
with $\overline{h}_{j,k}$ according to~\eqref{eq:convoluted_tightening_constant} below. 
Then, the chance constraints~\eqref{eq:chance_constraints} hold. 
\end{lemma}
\begin{proof}
Define $\mathrm{vec}([\tilde{G}_{0,k},\tilde{G}_{u,k}])=\tilde{\theta}_k$. 
Given that $\prob{\tilde{\theta}_k\in\Theta_{\delta,k}}\geq \delta$ is independent of $w_{[0,N-1]}$, $x_0$, the chance constraints~\eqref{eq:chance_constraints} hold if
\begin{align}
H_{\mathrm{x},j}^\top\begin{bmatrix}G_{0,k},G_{u,k}\end{bmatrix}\begin{bmatrix}\overline{x}_0\\u_{[0,k-1]}\end{bmatrix}+c_{\tilde{p}}\|H_{\mathrm{x},j}\|_{\Sigma_{\mathrm{x},k}}\leq 1, 
\end{align}
 for all $\tilde{\theta}_k\in\Theta_{\delta,k}$, $j\in\mathbb{I}_{[1,r]}$ with $\Sigma_{\mathrm{x},k}=G_{\mathrm{w},k} \mathrm{diag}_k(\Sigma_{\mathrm{w}})G_{\mathrm{w},k}^\top+G_{0,k}\Sigma_{\mathrm{x},0}G_{0,k}^\top$. 
For the mean trajectory $\overline{x}_k$, the properties of the Kronecker product yield 
\begin{align*}
&H_{\mathrm{x},j}^\top[\tilde{G}_{0,k},\tilde{G}_{u,k}]\begin{bmatrix}
\overline{x}_0\\u_{[0,k-1]}
\end{bmatrix}
=\left(\begin{bmatrix}
\overline{x}_0\\u_{[0,k-1]}
\end{bmatrix}^\top\otimes H_{\mathrm{x},j}^\top\right)\tilde{\theta}_k\\
\leq &\left\|\left(\begin{bmatrix}
\overline{x}_0\\u_{[0,k-1]}
\end{bmatrix}^\top \otimes H_{\mathrm{x},j}^\top\right)\Sigma_{\theta,k}^{1/2}\right\|\|\Sigma_{\theta,k}^{-1/2}\tilde{\theta}_k\|\\
\stackrel{\eqref{eq:ellipse}}{\leq} & 
\sqrt{\chi^2_{nkm+n^2}(\delta)}\left\|\Sigma_{\theta,k}^{1/2}\left(\begin{bmatrix}
\overline{x}_0\\u_{[0,k-1]}
\end{bmatrix}\otimes H_{\mathrm{x},j}\right)\right\|. 
\end{align*}
Similarly, for the variance we obtain
\begin{align*}
&\Sigma_{\mathrm{x},0}^{1/2}\tilde{G}_{0,k}^\top H_{\mathrm{x},j}
=(\Sigma_{\mathrm{x},0}^{1/2} \otimes H_{\mathrm{x},j}^\top)\vec{\tilde{G}_{0,k}}\\
=&(\Sigma_{\mathrm{x},0}^{1/2} \otimes H_{\mathrm{x},j}^\top)[I_{n^2},0_{n^2\times kmn}]\tilde{\theta}_k.
\end{align*}
Correspondingly, we get
\begin{align*}
&\|\Sigma_{\mathrm{x},k}^{1/2}H_{\mathrm{x},j}\| 
=\left\|\begin{bmatrix}\mathrm{diag}_k(\Sigma_{\mathrm{w}}^{1/2})&0\\0&\Sigma_{\mathrm{x},0}^{1/2}\end{bmatrix}
\begin{bmatrix}G_{\mathrm{w},k}^\top\\ G_{0,k}^\top\end{bmatrix}H_{\mathrm{x},j}\right\|\\
=&\left\|\begin{bmatrix}\mathrm{diag}_k(\Sigma_{\mathrm{w}}^{1/2})G_{\mathrm{w},k}^\top H_{\mathrm{x},j}\\\Sigma_{\mathrm{x},0}^{1/2}\hat{G}_{0,k}^\top H_{\mathrm{x},j}+(\Sigma_{\mathrm{x},0}^{1/2} \otimes H_{\mathrm{x},j}^\top)[I_{n^2},0_{n^2\times kmn}]\tilde{\theta}_k\end{bmatrix}\right\|\leq\overline{h}_{j,k}
\end{align*}
with 
\begin{align}
\label{eq:convoluted_tightening_constant}
 \overline{h}_{j,k}:=&\max_{\tilde{\theta}_k\in\Theta_{\delta,k}}\left\|\begin{bmatrix}\mathrm{diag}_k(\Sigma_{\mathrm{w}}^{1/2})G_{\mathrm{w},k}^\top H_{\mathrm{x},j}\\\Sigma_{\mathrm{x},0}^{1/2}\hat{G}_{0,k}^\top H_{\mathrm{x},j}+(\Sigma_{\mathrm{x},0}^{1/2} \otimes H_{\mathrm{x},j}^\top)[I_{n^2},0_{n^2\times kmn}]\tilde{\theta}_k\end{bmatrix}\right\|.
\end{align}
By combining these bounds, we obtain
\begin{align*}
&H_{\mathrm{x},j}^\top(G_{0,k}\overline{x}_0+G_{u,k}u_{[0,k-1]})+c_{\tilde{p}}\|H_{\mathrm{x},j}\|_{\Sigma_{\mathrm{x},k}}\\
\leq& H_{\mathrm{x},j}^\top(\hat{G}_{0,k}\overline{x}_0+\hat{G}_{u,k}u_{[0,k-1]})+c_{\tilde{p}}\overline{h}_{j,k}\\
+&\sqrt{\chi^2_{nkm+n^2}(\delta)}\left\|\Sigma_{\theta,k}^{1/2}\left(\begin{bmatrix}
\overline{x}_0\\u_{[0,k-1]}
\end{bmatrix}\otimes H_{\mathrm{x},j}\right)\right\|,
\end{align*}
for all $\tilde{\theta}_k\in\Theta_{\delta,k}$. 
Thus, Inequality~\eqref{eq:convoluted_tightening} implies~\eqref{eq:chance_constraints}. 
\end{proof}
For the derived mean and variance bound, there exist parameters $\tilde{\theta}_k\in\Theta_{\delta,k}$, such that the bounds hold individually with equality (although the combined bound~\eqref{eq:convoluted_tightening} is not tight). 
Computing the constants~\eqref{eq:convoluted_tightening_constant} requires maximizing a quadratic function over an ellipsoid, which is an SDP~\cite[Prop.~2.2]{yildirim2006minimum}, and a simple upper bound is given by:  
\begin{align}
\label{eq:convoluted_tightening_constant_alternative}
&\overline{h}_{j,k}\leq \sqrt{\chi^2_{nkm+n^2}(\delta)}\|(\Sigma_{\mathrm{x},0}^{1/2} \otimes H^\top_{\mathrm{x},j})[I_{n^2},0_{n^2\times kmn}]\Sigma_{\theta,k}^{1/2}\|\nonumber\\
&+\left\|H_{\mathrm{x},j}^\top 
\begin{bmatrix}G_{\mathrm{w},k},\hat{G}_{0,k}\end{bmatrix}
\begin{bmatrix}\mathrm{diag}_k(\Sigma_{\mathrm{w}}^{1/2})&0\\0&\Sigma_{\mathrm{x},0}^{1/2}\end{bmatrix}\right\|.
\end{align}
Given the parameter estimates $\hat{\theta}_k$ with variance $\Sigma_{\theta,k}$ (Lemma~\ref{lemma:id_correlated}) and constants $\overline{h}_{j,k}$~\eqref{eq:convoluted_tightening_constant} for each prediction step $k\in\mathbb{I}_{[1,N]}$, we can pose the following problem: 
\begin{align}
\label{eq:stoch_opt_convoluted_robust}
&\min_{u_{[0,N-1]}\in\mathbb{U}^N}~\sum_{k=0}^{N-1}\|\overline{x}_{k+1}\|_Q^2+\|u_k\|_R^2\\
\text{s.t. }&H_{\mathrm{x},j}^\top\overline{x}_k+ \sqrt{\chi^2_{nkm+n^2}(\delta)}\left\|\Sigma_{\theta,k}^{1/2}\begin{bmatrix}
\overline{x}_0\\u_{[0,k-1]}
\end{bmatrix}\otimes H_{\mathrm{x},j}\right\|\nonumber\\
&\leq 1-c_{\tilde{p}}\overline{h}_{j,k},~ j\in\mathbb{I}_{[1,r]},~k\in\mathbb{I}_{[0,N]},\nonumber\\
&\overline{x}_{k}=\hat{G}_{0,k}\overline{x}_0+\hat{G}_{\mathrm{u},k}u_{[0,k-1]},~k\in\mathbb{I}_{[1,N]}.\nonumber 
\end{align}
\begin{corollary}
Suppose the conditions in Lemma~\ref{lemma:id_correlated} hold. 
Problem~\eqref{eq:stoch_opt_convoluted_robust} is a QP subject to second-order cone constraints. 
Any feasible input sequence $u_{[0,N-1]}$ for Problem~\eqref{eq:stoch_opt_convoluted_robust} satisfies the chance constraints~\eqref{eq:chance_constraints}.
\end{corollary}
\begin{proof}
The chance constraints~\eqref{eq:chance_constraints} follow directly from Lemmas~\ref{lemma:id_correlated} and \ref{lemma:convoluted_robust}.
This result exploits that we have individual chance constraints~\eqref{eq:chance_constraints} for each $k\in\mathbb{I}_{[1,N]}$ and hence do \textit{not} require the joint distribution of $(\theta_k,\theta_{k+1})$. 
\end{proof}
Problem~\eqref{eq:stoch_opt_convoluted_robust} addresses the chance constraints. 
A performance bound may require a modified cost~\cite{furieri2021near,huang2021robust}.

\subsection{Existing approaches using multi-step predictor bounds} 
\label{sec:data_driven_4}
In the following, we detail some approaches from the literature that, similar to Lemma~\ref{lemma:convoluted_robust}, directly exploit parametric error bounds on the multi-step predictor.

\subsubsection*{Impulse response models}
Early MPC implementations, especially in process control, are focused on impulse/step response models~\cite{richalet1978model,cutler1980dynamic,lundstrom1995limitations}. 
Corresponding \textit{robust} designs for FIR models with parametric uncertainty have also been derived using $\min-\max$ problems~\cite{zheng1993robust}. 
By additionally considering set-membership estimation to identify the parametric error, a robust data-driven/adaptive MPC for FIR models is derived in~\cite{tanaskovic2014adaptive}. 
For a more general parametrization with OBFs~\cite{heuberger1995generalized}, and assuming only output measurement noise, robust/chance constraints under parametric uncertainty can also be enforced without conservatism~\cite{heirung2017dual,soloperto2019dual}. 
Notably, the parameter dimension with OBF can be significantly smaller, which reduces the computational complexity of the vertex enumeration.

\subsubsection*{Data-driven/enabled predictive control}
The direct identification of multi-step predictors~\eqref{eq:model_convoluted_split} has already been proposed in SPC~\cite{favoreel1999spc}.  
More recently, a variation of this approach known as \textit{data-driven/enabled predictive control} has received significant attention~\cite{berberich2020data,coulson2021distributionally}. 
The basic idea is to specify an implicit model by stacking the measured data in a Hankel matrix, which is directly used within the optimization problem, instead of the usual sequential model identification and predictive control. 
In~\cite{huang2021robust,coulson2021distributionally}, it is shown that this implicit model also allows for simple reformulations of the $\min-\max$/distributionally-robust problem, similar to Lemma~\ref{lemma:convoluted_robust}.  
In a related result, \cite{yin2021maximum,yin2021data} provide an MLE with simple bounds on the prediction. 
Notably, most of these results focus on measurement noise instead of disturbances, compare the discussion in Section~\ref{sec:SysID_3} on system identification.  
When comparing data-driven/enabled predictive control approaches~\cite{coulson2021distributionally,berberich2020data,huang2021robust} with standard model-based MPC, we note that the difference is \textit{not} only attributable to the fact that a \textit{direct} data-driven design is used (cf.~\cite{dorfler2021bridging} for a detailed discussion on potential benefits); but also due to the fact that \textit{multi-step predictors} are (implicitly) specified.

\subsubsection*{Set-membership estimation} In~\cite{terzi2019learning,lauricella2020data,terzi2022robust}, set-membership estimation is used for
multi-step predictors of the form~\eqref{eq:model_convoluted_split}. 
These approaches consider polytopic bounded noise and disturbances, resulting in polytopic sets $\tilde{\theta}_k\in\Theta_{\delta,k}$.  
Prediction bounds can then be directly obtained with linear inequality constraints (cf.~\cite{lauricella2020data}), which are guaranteed to be less conservative than corresponding bounds from a state space model~\cite[Thm.~1]{terzi2022robust}. 
In contrast to Problem~\eqref{eq:stoch_opt_convoluted_robust}, the resulting MPC approaches over-approximate the prediction error with a constant bound to allow for simpler robust MPC methods~\cite{terzi2019learning} and ensure recursive feasibility~\cite{terzi2022robust}.  

\subsubsection*{System level synthesis (SLS)}
SLS is a method to jointly optimize over feedback policies by parametrizing the problem using the (closed-loop) multi-step predictors $G_{\mathrm{u}}$,$G_0$, $G_{\mathrm{w}}$, which are called \textit{system responses}~\cite{chen2020robust}, compare also the Input-Output Parametrization (IOP)~\cite{furieri2021near}.  
In~\cite{furieri2021near}, multi-step predictors~\eqref{eq:model_convoluted} are directly identified and a simple constraint tightening, similar to Lemma~\ref{lemma:convoluted_robust}, is obtained.  
In~\cite{bujarbaruah2020robust,chen2020robust}, the authors consider parametric errors in the state space formulation and then (partially) over-approximate it with a parametric uncertainty on the multi-step predictors~\eqref{eq:model_convoluted} to allow for a similarly simple robustification.  
Specifically, similar to Lemma~\ref{lemma:convoluted_robust}, the constraint tightening in~\cite[Equ.~(13)]{bujarbaruah2020robust} is proportional to a weighted norm of 
$\begin{bmatrix}
\overline{x}_0\\u_{[0,k-1]}
\end{bmatrix}$.

\subsection{Discussion}
\label{sec:data_driven_3} 
\subsubsection*{Complexity and conservatism}
 For both model parametrizations, the considered derivation (Sec.~\ref{sec:data_driven_1}/\ref{sec:data_driven_2}) introduces conservatism by: a) separately bounding the variance and mean (which allows for a pre-computation of the variance related terms); b) using $\prob{\tilde{\theta}_k\in\Theta_{\delta,k}}\geq \delta$ (Lemma~\ref{lemma:id_correlated}) instead of directly working with the distribution $\theta_k\in\mathcal{N}(\hat{\theta}_k,\Sigma_{\theta,k})$. 
Apart from this simplification, the derivation in Lemma~\ref{lemma:convoluted_robust} is tight. 
Another major benefit of Problem~\eqref{eq:stoch_opt_convoluted_robust} is its simplicity, i.e., compared to Problem~\eqref{eq:stoch_opt_convoluted} we only require the following changes:
\begin{itemize}
\item The probability level is increased with $\tilde{p}>p$.
\item The constraint is further tightened by a scaled norm of the vector 
$\Sigma_{\theta,k}^{1/2}\begin{bmatrix}
\overline{x}_0\\u_{[0,k-1]}
\end{bmatrix}\otimes H_{\mathrm{x},j}\in\mathbb{R}^{nkm+n^2}$. 
\item Constants $\overline{h}_{j,k}\geq \|H_{\mathrm{x},j}\|_{\Sigma_{\mathrm{x},k}}$ can be computed offline using~\eqref{eq:convoluted_tightening_constant} or \eqref{eq:convoluted_tightening_constant_alternative}.
\end{itemize} 
Regarding the computational complexity, the additional vector-norm in the constraint tightening results in second-order cone constraints. 
In case of polytopic parameter sets $\Theta_{\delta,k}$, a similar derivation results in a linearly constrained QP. 

Considering the state space problem~\eqref{eq:stoch_opt_robust}, standard tube-based approaches to robustly account for parametric uncertainty (cf.~\cite{kouvaritakis2016model,primbs2009stochastic,fleming2014robust,lorenzen2019robust,Koehler2019Adaptive,lu2021robust}) also result in a QP subject to linear or second-order-cone constraints, depending if the set $\Theta$ is a polytope or ellipsoid. 
Notably, these approaches retain the sparsity structure. 
Tube-based methods require an additional offline step to design a suitable polytope/ellipsoid for the tube parametrization. 
Considering specifically the polytopic setting, a more complex polytope may reduce the conservatism but significantly increases the computational complexity (cf.~\cite[Tab.~5.2]{kouvaritakis2016model}), thus resulting in a non-trivial offline tuning problem. 
In addition, there exist degrees of freedom in the tube propagation that require a trade-off between conservatism and computational complexity (cf.~\cite[Tab.~1]{Koehler2019Adaptive}). 
The sequential nature of the tube-propagation with its fixed parametrization can result in significant conservatism, which is difficult to quantify a-priori.
Less conservative/tight reachability bounds for state space models with parametric uncertainty require more sophisticated robust control tools, e.g., integral quadratic constraints. 
However, exploiting such tools in MPC is part of ongoing research~\cite{schwenkel2021model}. 

\subsubsection*{Receding horizon implementation}
One of the core principles of MPC is generating feedback by repeatedly solving open-loop optimization problems~\cite{mayne2000constrained,kouvaritakis2016model,rawlings2017model}.
Thus, \textit{recursive feasibility} is of paramount importance, which can be ensured with well-established methods for state space formulations~\cite{mayne2000constrained,kouvaritakis2016model,rawlings2017model}. 
For the following discussion on recursive feasibility, we restrict ourselves to the \textit{robust} setting, as the unbounded disturbances in stochastic MPC require extra care~\cite{primbs2009stochastic,farina2016stochastic,hewing2020recursively}.
A crucial feature of the sequential disturbance propagation in tube-based MPC (cf.~\cite{kouvaritakis2016model,fleming2014robust,lorenzen2019robust,Koehler2019Adaptive,lu2021robust}) is the fact that recursive feasibility directly holds, assuming a proper parametrization and terminal set constraint. 
On the other hand, directly utilizing the multi-step predictor bounds (cf. Lemma~\ref{lemma:convoluted_robust}), does in general \textit{not} ensure recursive feasibility. 
This issue can be addressed using: shrinking/adaptive horizons or a multi-rate/multi-step implementations~\cite{bujarbaruah2020robust,chen2020robust,terzi2022robust}; intersecting different over-approximations of the reachable set~\cite{maddalena2021kpc,lauricella2020data}; 
considering FIR/OBF models~\cite{tanaskovic2014adaptive,heirung2017dual,soloperto2019dual} (cf. also Volterra series~\cite{gruber2013computationally,genecili1995design}). 
This issue arises partially due to incompatible models, e.g., $\hat{G}_{0,k}\hat{G}_{0,k}\neq \hat{G}_{0,2k}$, and has similarities to the known problems associated with move-blocking MPC~\cite{gondhalekar2010least}.
Due to the different closed-loop implementation, it is not a-priori clear if the reduced \textit{open-loop} conservatism of multi-step predictors (Sec.~\ref{sec:data_driven_2}) results in improved \textit{closed-loop} performance.

\subsubsection*{Input-output (IO) models}
The assumption of noisy state measurements $\tilde{x}_k$ can be relaxed to noisy output measurements $\tilde{y}_k=y_k+\eta_k\in\mathbb{R}^{n_{\mathrm{y}}}$. 
For example, assuming (final state) observability~\cite[Def.~4.29]{rawlings2017model} (e.g., ARX model), a non-minimal state is given by $x_k=(y_{[k-\nu,k-1]},u_{[k-\nu,k-1]})\in\mathbb{R}^{\nu(m+n_{\mathrm{y}})}$ with the lag $\nu\leq n$~\cite{coulson2021distributionally}. 
Hence, a noisy state measurement $\tilde{x}_k=x_k+\epsilon_k$ is directly available with the correlated(!) noise $\epsilon_k=(\eta_{[k-\nu,k-1]},0_{ m\nu\times 1})$. The correlation requires a modification of the MLE (Sec.~\ref{sec:SysID}), cf.~\cite{yin2021maximum,yin2021data}.
The MLE can also be directly applied to noisy output measurements using a Kalman filter~\cite{aastrom1971system,gibson2005robust}. 
The fact that the resulting state space model is not necessarily \textit{minimal} can complicate the offline design required in tube-based state space approaches~\cite{kouvaritakis2016model,fleming2014robust,lorenzen2019robust,Koehler2019Adaptive,lu2021robust}.
On the other hand, a minimal state space representation is \textit{not} directly relevant for multi-step predictors and most of these approaches consider the IO setting~\cite{tanaskovic2014adaptive,heirung2017dual,soloperto2019dual,berberich2020data,coulson2021distributionally,huang2021robust,yin2021maximum,yin2021data,iannelli2021experiment,terzi2019learning,lauricella2020data,terzi2022robust,furieri2021near,favoreel1999spc}.
It is tempting to distinguish "standard" MPC schemes~\cite{mayne2000constrained,kouvaritakis2016model,rawlings2017model} and more recent "data-driven" MPC methods~\cite{berberich2020data,coulson2021distributionally} in terms of input-state vs. IO setting. 
However, based on the exposition in this paper, we postulate that a major difference is the model parametrization: sequential state space models vs. (implicit) direct multi-step predictors. 

\subsubsection*{Nonlinear systems}
While we focus on linear systems throughout the paper, similar conclusions can be drawn for nonlinear systems, where some of the differences become even more pronounced. 
Nonlinear multi-step predictors, i.e., $x_{i+k}=f_k(x_i,u_{[i,i+k-1]},w_{[i,i+k-1]})$, can be identified using linearly parametrized models, e.g., Volterra series~\cite{genecili1995design,gruber2013computationally}, or non-parametric approaches, e.g., kernel methods~\cite{maddalena2021kpc,thorpe2022data}, 
analogous to nonlinear state-space models.\footnote{%
The direct identification of the impulse response used in~\eqref{eq:model_convoluted_split} can also be viewed as a \textit{non-parametric} problem, especially for $N\rightarrow\infty$~\cite{aastrom1971system}.} 
The fact that such models do \textit{not} require any conservative/complex disturbance propagation is even more critical in the nonlinear case. 
However, the complexity issues of multi-step predictors for large prediction horizons $N$ are also amplified for nonlinear problems. 
Furthermore, the difference in the prior information for state space models and multi-step predictors are more important in the nonlinear setting.

\section{Conclusion}
\label{sec:conclusion}
We have provided a tutorial-style exposition of data-driven stochastic predictive control using state space models or multi-step predictors.
In particular, we have investigated the challenges associated with parametric uncertainty in both model parametrizations. 
Tube-based methods for state space models need to trade-off computational complexity and conservatism, and require additional offline designs with various free design parameters. 
On the other hand, we derived simple (partially tight) bounds for multi-step predictors (Lemma~\ref{lemma:convoluted_robust}), which do not suffer from similar conservatism.
However, an a-priori comparison of the \textit{closed-loop} performance in a receding horizon implementation is challenging due to the required modifications with multi-step predictors (e.g., multi-rate implementation). 
A direct comparison of the computational complexity is also difficult, as even within tube-based approaches the complexity can vary by orders of magnitude (cf.~\cite[Tab.~1]{Koehler2019Adaptive}). 
Nonetheless, tube based approaches retain the sparse structure of the state space formulation and hence result in a reduced computational complexity for long horizons $N$ (cf.~\cite[Sec.~3]{kouzoupis2018recent}). 
In addition to these differences, main factors that should guide the choice of parametrizations are: 
\begin{itemize}
\item magnitude of the parametric uncertainty;
\item 
simplicity in receding horizon implementation;
\item 
scalability for long prediction horizons;
\item 
different model priors and the prevalence of disturbances or noise in the parameter identification.
\end{itemize}  
\subsubsection*{Open-issues}
Lemma~\ref{lemma:convoluted_robust} could be improved by directly using $\theta_k\sim\mathcal{N}(\hat{\theta}_k,\Sigma_{\theta,k})$.  
Numerical comparisons regarding computational complexity and closed-loop performance would be beneficial. 
Given the respective benefits, unifying the parametrizations in a hybrid model structure is a promising research direction (cf., \textit{DiRec/DIRMO} strategies~\cite{taieb2012review}).

\bibliographystyle{ieeetran}  
\bibliography{Literature}  

\end{document}